\documentclass[preprint,12pt,authoryear]{elsarticle}

\usepackage{amsmath}
\usepackage{amssymb}
\usepackage{amsthm}
\usepackage{booktabs}
\usepackage{float}
\usepackage{graphicx}
\usepackage{mathtools}
\usepackage{enumitem}
\usepackage{microtype}
\usepackage{tikz}
\usetikzlibrary{arrows.meta,positioning}

\journal{European Journal of Operational Research}

\newtheorem{assumption}{Assumption}
\newtheorem{definition}{Definition}
\newtheorem{proposition}{Proposition}
\newtheorem{remark}{Remark}
\newtheorem{example}{Example}

\newcommand{\E}{\mathbb{E}}
\newcommand{\R}{\mathbb{R}}

\newcommand{\cP}{\mathcal{P}}

\newcommand{\CVaR}{\operatorname{CVaR}}

\newcommand{\argmax}{\operatorname*{arg\,max}}
\newcommand{\argmin}{\operatorname*{arg\,min}}
\newcommand{\pos}[1]{\left(#1\right)^+}

\begin{document}

\begin{frontmatter}

\title{A Finite-Candidate Distributionally Robust Tri-Objective Newsvendor
Model for Energy Storage Capacity Reservation}

\author[aff1]{Nuerxiati Abudurexiti}
\affiliation[aff1]{
  organization={Department of Management Science and Engineering},
  addressline={Xinjiang University},
  city={Urumqi},
  postcode={830046},
  country={China}
}

\begin{abstract}
Energy storage operators often reserve usable capacity before uncertain market
opportunities are realized, for example when a day-ahead operator commits
capacity for an evening peak-spread or ancillary-service window. Price spikes,
renewable forecast errors, activation calls and imbalance penalties can make the
decision-relevant opportunity distribution asymmetric and heavy-tailed rather
than Gaussian. The key difficulty is not only non-normality, but also
distributional misspecification: a capacity calibrated to one opportunity
distribution may generate severe tail losses or large regret under another
plausible distribution. This paper develops a finite-candidate distributionally
robust tri-objective storage-newsvendor framework for this capacity-reservation
decision. The model jointly maximizes worst-case expected profit, minimizes
worst-case CVaR tail loss and minimizes maximum regret over candidate non-normal
distributions. We derive a critical-fractile expected-profit benchmark and a
finite-scenario quadratic programming formulation that solves the three
distributionally robust objectives together. The empirical and numerical results
show that capacity decisions are sensitive to distributional shape, not only to
mean and variance, and that the three robust objectives generate different but
economically interpretable reservation policies.
\end{abstract}

\begin{highlights}
\item We formulate energy-storage capacity reservation as a finite-candidate distributionally robust newsvendor-type decision problem.
\item We develop a finite-candidate distributionally robust tri-objective model: max worst-case expected profit, min worst-case CVaR and min max regret.
\item We derive a finite-scenario quadratic programming form for the joint distributionally robust model.
\end{highlights}


\end{frontmatter}

\section{Introduction}
\label{sec:introduction}

Energy storage plays an increasingly important role in electricity markets,
especially when renewable generation, imbalance management and ancillary-service
products create short-lived operating opportunities. Consider a storage operator
that must reserve \(q\) MWh of usable capacity before an evening operating window:
the realized opportunity may come from an unusually large price spread,
ancillary-service activation, renewable forecast error or imbalance penalty.
Reserving too much capacity leaves idle capacity and degradation cost, whereas
reserving too little capacity misses scarcity-value or flexibility revenue. Thus,
the capacity reservation problem has a natural asymmetric trade-off between
excess capacity and insufficient capacity.

Figure~\ref{fig:storage_reservation_mechanism} summarizes this running example.
The operator commits capacity before observing the effective opportunity
\(\xi\). After \(\xi\) is realized, the same decision may generate utilized
capacity revenue, idle-capacity salvage value or missed-opportunity loss.

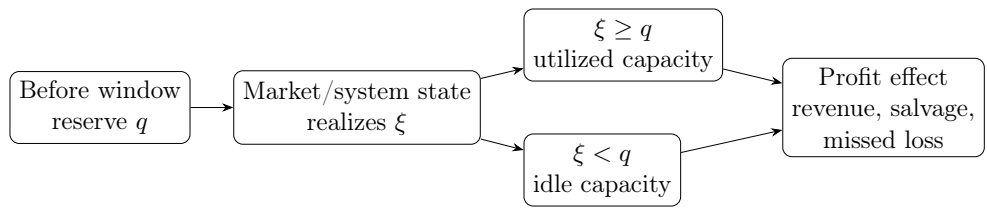
\begin{figure}[H]
\centering
\resizebox{0.95\textwidth}{!}{%
\begin{tikzpicture}[
  node distance=0.65cm and 0.65cm,
  >=Stealth,
  every node/.style={align=center,font=\small},
  box/.style={draw,rounded corners,minimum width=2.45cm,minimum height=0.85cm},
  outcome/.style={draw,rounded corners,minimum width=2.30cm,minimum height=0.75cm}
]
\node[box] (decision) {Before window\\reserve \(q\)};
\node[box,right=of decision] (uncertainty) {Market/system state\\realizes \(\xi\)};
\node[outcome,right=of uncertainty,yshift=0.95cm] (use) {\(\xi\ge q\)\\utilized capacity};
\node[outcome,right=of uncertainty,yshift=-0.95cm] (idle) {\(\xi<q\)\\idle capacity};
\node[box,right=of use,xshift=0.25cm,yshift=-0.95cm] (payoff) {Profit effect\\revenue, salvage,\\missed loss};
\draw[->] (decision) -- (uncertainty);
\draw[->] (uncertainty) -- (use);
\draw[->] (uncertainty) -- (idle);
\draw[->] (use) -- (payoff);
\draw[->] (idle) -- (payoff);
\end{tikzpicture}
}
\caption{The reserved capacity \(q\) is chosen before the effective storage
opportunity \(\xi\) is observed, creating a newsvendor-type overage--underage
trade-off. The figure anchors the paper's running example and the profit function
used in the model.}
\label{fig:storage_reservation_mechanism}
\end{figure}

The uncertainty behind this trade-off is rarely normal. Electricity prices and
flexibility opportunities are affected by scarcity events, renewable forecast
errors, extreme weather, market rules and system imbalance. Empirical studies on
electricity prices document spikes, abrupt market-state shifts, extreme
volatility and heavy tails~\citep{Weron2009,Weron2014}. Energy-storage operation has also been studied
as a stochastic operational problem in settings such as stationary storage
management~\citep{WeitzelGlock2018}, rolling-forecast
operation~\citep{GhadimiPowell2024}, and frequency-regulation
participation~\citep{Lauinger2024}. These studies show that storage decisions are strongly
affected by market uncertainty. However, most of this literature focuses on
multi-period dispatch, bidding or dynamic control. The single-period
pre-commitment question, namely how much capacity should be reserved before a
non-normal market opportunity is realized, has received much less explicit
analytical treatment.

\subsection*{Positioning in related work}

This paper is positioned against four related streams. First, the storage
operation literature has studied stationary storage management, rolling-forecast
operation and frequency-regulation participation~\citep{WeitzelGlock2018,
GhadimiPowell2024,Lauinger2024}. These studies motivate the market setting, but
their main focus is dynamic dispatch or bidding rather than a single-period
capacity-reservation layer.

Second, the newsvendor literature provides the basic asymmetric pre-commitment
structure. The surveys of \citet{Khouja1999} and \citet{Qin2011} showed that the
newsvendor model had been extended far beyond retail inventory.
\citet{PetruzziDada1999} further studied price-inventory interaction. This stream
establishes the relevance of the newsvendor structure, but it does not directly
model storage opportunities generated by price spreads, reserve activation or
imbalance events.

Third, robust and regret-based newsvendor models address distributional
misspecification. For example, \citet{GallegoMoon1993} developed distribution-free newsvendor
models based on partial information, and \citet{PerakisRoels2008} studied
partial-information decisions through maximum regret. More recent robust and
distributionally robust models, including \citet{XuZhengJiang2022} and
\citet{Fu2024}, constructed ambiguity sets from nonparametric or
stochastic-dominance information. We adapt this logic to storage capacity
reservation by treating candidate non-normal opportunity distributions as the
source of regret.

Fourth, risk-averse and data-driven newsvendor studies motivate the CVaR and
sample-based components. CVaR has a tractable optimization
representation~\citep{RockafellarUryasev2000}, and inventory studies have used
CVaR or coherent risk measures to control tail outcomes~\citep{GotohTakano2007,
AhmedCakmakShapiro2007}. Sample-average and contextual newsvendor studies have
also shown that finite-sample performance depends on the distributional behavior
near the critical quantile~\citep{Levi2015,Lin2022,BanRudin2019,Huber2019,
Erkip2023,Serrano2024,Sadana2025}. These results are especially relevant for
storage because rare price spikes and heavy tails can make empirical critical
quantiles unstable.

The purpose of this paper is to develop a tractable finite-candidate
distributionally robust newsvendor-type model for energy-storage capacity
reservation under non-normal uncertainty and distributional misspecification.
Different from generic newsvendor models, the
uncertain quantity in this paper is an effective storage opportunity generated by
market prices, ancillary-service calls, renewable forecast errors and imbalance
conditions. Different from full multi-period storage-operation models, the focus
is a transparent single-period capacity-reservation layer. We use the
evening-window running example to define the capacity decision, the effective
opportunity and the asymmetric payoff terms. On this common economic primitive,
we formulate one finite-candidate distributionally robust tri-objective decision
problem that simultaneously
maximizes worst-case expected profit, minimizes worst-case CVaR tail loss and
minimizes maximum regret under distributional ambiguity.

The main contributions of this paper are as follows. First,
Section~\ref{sec:unified_framework} formulates energy-storage capacity
reservation as a newsvendor-type decision problem and derives a profit function
that contains utilized-capacity revenue, reservation cost, salvage value and
missed-opportunity penalty. Second,
Section~\ref{subsec:multi_objective} formulates a finite-candidate
distributionally robust tri-objective optimization problem that puts worst-case
expected-profit maximization, worst-case CVaR minimization and maximum-regret
minimization into one Pareto decision model.
Third, Section~\ref{sec:closed_form_solutions} shows that the expected-profit
component admits a critical-fractile solution, so non-normal uncertainty affects
the decision through the relevant quantile rather than through normality itself.
Fourth, Section~\ref{sec:finite_scenario_solution} derives a finite-scenario
quadratic programming formulation that solves the three objectives jointly rather
than as three disconnected optimization problems. Fifth,
Section~\ref{sec:numerical_analysis} uses public RTE aFRR data and EM-fitted
log-NMVM candidates to show that the normal approximation can severely distort
the empirical critical capacity. It then evaluates finite-candidate robust
expected-profit, CVaR, regret and multi-objective policies on the empirical
candidate set.

The remainder of the paper is organized as follows.
Section~\ref{sec:unified_framework} develops the unified storage newsvendor
framework, including the opportunity variable, the profit function, and the
finite-candidate distributionally robust expected-profit--CVaR--regret
formulation and its finite-scenario
solution form.
Section~\ref{sec:closed_form_solutions} derives closed-form benchmark solutions
and analytical implications of non-normality. Section~\ref{sec:numerical_analysis}
presents the numerical analysis.

\section{A unified newsvendor framework for storage capacity reservation}
\label{sec:unified_framework}

The central methodological contribution of this paper is to represent storage
capacity reservation through a single newsvendor primitive: a pre-committed
capacity \(q\), a random effective opportunity \(\xi\), and a profit function
\(\Pi(q,\xi)\) that captures both excess-capacity and insufficient-capacity
effects. On this primitive, the storage operator faces a finite-candidate
distributionally robust multi-objective decision problem: it wants high
worst-case expected profit, low worst-case tail loss, and low regret over a
finite ambiguity set of plausible opportunity distributions. This section
develops the unified framework and shows how the three distributionally robust
objectives can be formulated and solved in a tractable way.

\subsection{Problem setting and effective opportunity}
\label{subsec:problem_setting}

Consider the running example introduced in the Introduction. A storage operator
decides before an evening operating window how much usable capacity to reserve for
a possible high-value market opportunity. The same notation also applies to energy
arbitrage, ancillary services, imbalance reduction, emergency flexibility or
demand response. Let
\[
q\in[0,Q]
\]
denote the reserved usable capacity, where \(Q>0\) is the physical or market
participation limit. The unit of \(q\) may be MWh for energy capacity or MW for
power capacity, depending on the market product.

The variable \(q\) should be understood as a pre-committed amount of usable storage
resource. For an energy-arbitrage product, \(q\) represents the energy that can be
charged or discharged during the operating window. For a regulation or reserve
product, \(q\) represents the power capacity that is made available to the market.
When the product has a fixed service duration, a power commitment can be converted
into an energy-equivalent commitment by multiplying by that duration. Therefore,
the model uses a single scalar \(q\) to represent the capacity committed before
uncertainty is realized, while allowing the economic parameters below to reflect
the specific market product.

Let \(\xi\ge 0\) denote the random effective opportunity size in the future
operating window. It aggregates market and system conditions into the amount of
capacity that can be profitably used. For example,
\begin{equation}
\xi=\psi(\Delta p,d^{AS},w,l,\tau),
\label{eq:xi_mapping}
\end{equation}
where \(\Delta p\) is a price spread or price-spike intensity, \(d^{AS}\) is
ancillary-service activation demand, \(w\) is renewable output or forecast error,
\(l\) is load, and \(\tau\) represents calendar, weather, or market-state
information. Equation~\eqref{eq:xi_mapping} is a reduced-form mapping: the
function \(\psi(\cdot)\) converts these primitives into an effective capacity
opportunity.

The variable \(\xi\) is therefore not a direct forecast of one physical quantity. It
is a reduced-form representation of the market opportunity that can absorb the
reserved capacity. If \(\xi>q\), the market would have been able to use more
capacity than the operator reserved, and the difference \(\xi-q\) represents a
missed opportunity. If \(\xi<q\), only \(\xi\) units of the reserved capacity are
profitably used, and \(q-\xi\) represents idle or salvageable capacity. In this
sense, the pair \((q,\xi)\) plays the same role as order quantity and demand in a
classical newsvendor model, but \(\xi\) is generated by electricity-market and
system-operation primitives rather than by consumer demand alone.

Throughout the paper, the opportunity distribution is not restricted to the
normal family. In the benchmark models, \(F\) denotes the distribution function
of \(\xi\). This distribution is not assumed to be Gaussian; in particular, it
may be asymmetric or heavy-tailed because storage opportunities are often driven
by price spikes and other tail events. In the robust and regret-based models, the
operator may not know a unique distribution \(F\); instead, the available
information is represented by an ambiguity set \(\cP\) of plausible
distributions. Thus, \(F\) is used when the opportunity distribution is specified,
whereas \(\cP\) is used when distributional ambiguity is explicitly modeled.

\subsection{A flexible NMVM opportunity specification}
\label{subsec:nmvm_opportunity}

The analysis below does not require a unique parametric generator for \(\xi\).
Nevertheless, a normal mean--variance mixture (NMVM) representation is a useful
way to describe skewed and heavy-tailed storage opportunities without imposing
Gaussianity. In this paper, the moment-matched normal distribution is used only
as a benchmark because it summarizes \(\xi\) by \((\E[\xi],\operatorname{Var}(\xi))\).
An NMVM specification is more decision-relevant for storage capacity reservation:
through the mixing variable \(Z\), it can represent low-stress, normal-stress and
scarcity-stress market regimes, as well as continuous tail-thickening
mechanisms, while preserving the nonnegative scale of effective capacity
opportunities after the transformation below. Let \(N\sim N(0,1)\), and let
\(Z\ge 0\) be a nonnegative mixing
variable independent of \(N\). Throughout the NMVM specification, we require the
first two moments of the mixing variable to be finite,
\begin{equation}
\E[Z]<\infty,
\qquad
\E[Z^2]<\infty.
\label{eq:z_moment_condition}
\end{equation}
A latent opportunity driver may be written as
\begin{equation}
Y\overset{d}{=}\mu_{\xi}+\gamma_{\xi}Z+\sigma_{\xi}\sqrt{Z}N,
\qquad \sigma_{\xi}>0,
\label{eq:nmvm_latent}
\end{equation}
and the effective opportunity is generated by a nondecreasing map
\begin{equation}
\xi=g(Y),\qquad g:\R\rightarrow\R_+.
\label{eq:positive_mapping}
\end{equation}
For example, \(g(y)=e^y\) gives a positive opportunity variable, while a bounded
logistic map can impose a physical or market participation limit.
Equations~\eqref{eq:nmvm_latent} and~\eqref{eq:positive_mapping} separate the
latent non-normal driver from the nonnegative storage-opportunity scale.

The choice of the mixing variable \(Z\) is a substantive modeling decision, not
a harmless numerical detail. In an electricity-storage application, \(Z\) should
represent the latent intensity of market stress, price-spike volatility,
ancillary-service activation volatility, or imbalance-settlement stress. Different
choices of \(Z\) imply different skewness and tail behavior. Conditional on
\(Z\), the latent driver is Gaussian; hence, when \(g\) is strictly increasing,
\begin{equation}
F_{\xi}(x)
=
\E\left[
\Phi\left(
\frac{g^{-1}(x)-\mu_{\xi}-\gamma_{\xi}Z}
{\sigma_{\xi}\sqrt{Z}}
\right)
\right],
\label{eq:nmvm_cdf}
\end{equation}
for \(x>0\).

\begin{example}[Finite stress-regime log-NMVM]
\label{ex:discrete_nmvm}
A simple and interpretable special case is a finite stress-regime mixture. Let
\[
Z\in\{z_1,z_2,z_3\},\qquad
\Pr\{Z=z_j\}=p_j,\quad
\sum_{j=1}^{3}p_j=1.
\]
When \(0<z_j<\infty\) for all \(j\), the moment condition
\eqref{eq:z_moment_condition} is automatically satisfied.
The three states can be interpreted as low, normal and high opportunity-stress
regimes. Under the positive mapping \(g(y)=e^y\), the opportunity distribution
becomes a three-component lognormal mixture:
\begin{equation}
F_{\xi}(x)
=
\sum_{j=1}^{3}
p_j
\Phi\left(
\frac{\log x-\mu_{\xi}-\gamma_{\xi}z_j}
{\sigma_{\xi}\sqrt{z_j}}
\right),
\qquad x>0.
\label{eq:discrete_nmvm_cdf}
\end{equation}
Example~\ref{ex:discrete_nmvm} is useful when \(Z\) is interpreted as a small
number of market-stress regimes rather than as a fully specified continuous
mixing law.
\end{example}

If \(M_Z(\cdot)\) denotes the moment generating function of \(Z\),
then, for values of \(t\) in its domain,
\begin{equation}
\E[e^{tY}]
=
e^{t\mu_{\xi}}
M_Z\left(t\gamma_{\xi}+\frac{1}{2}t^2\sigma_{\xi}^2\right).
\label{eq:nmvm_mgf}
\end{equation}
Equations~\eqref{eq:nmvm_cdf} and~\eqref{eq:nmvm_mgf} show how an NMVM
specification can be calibrated once a defensible law for \(Z\) is selected.

\begin{example}[Continuous NMVM candidates: VG and GH]
\label{ex:continuous_nmvm}
Two continuous choices of \(Z\) are especially relevant. If
\(Z\sim \operatorname{Gamma}(a_Z,\theta_Z)\), then \(Y\) has a variance-gamma type
mixture and
\[
M_Z(u)=(1-\theta_Z u)^{-a_Z},\qquad u<1/\theta_Z.
\]
If \(Z\sim \operatorname{GIG}(\lambda,\chi,\psi)\) with
\(\lambda\ne -1/2\), then \(Y\) has a GH mixture. In this case,
\[
M_Z(u)
=
\left(\frac{\psi}{\psi-2u}\right)^{\lambda/2}
\frac{
K_{\lambda}\!\left(\sqrt{\chi(\psi-2u)}\right)
}{
K_{\lambda}\!\left(\sqrt{\chi\psi}\right)
},
\qquad u<\psi/2,
\]
where \(K_{\lambda}(\cdot)\) is the modified Bessel function of the second kind.
For these continuous mixtures, the maintained requirement is the finite-moment
condition in~\eqref{eq:z_moment_condition}. The remaining shape parameters
should be estimated, calibrated, or used in sensitivity analysis rather than
fixed without evidence. The two choices support log-VG and log-GH opportunity distributions
when \(g(y)=e^y\), provided the exponential-moment domain contains the values
needed for the first two moments of \(\xi\)
\citep{McNeilFreyEmbrechts2015,BarndorffNielsen1997,MadanSeneta1990}.
\end{example}

This paper keeps NMVM as the main non-normal specification and treats the choice
of \(Z\) as part of the modeling design. The numerical section therefore compares
degenerate, finite-regime, VG and GH versions of the same NMVM opportunity
framework. The GH candidate used below has \(\lambda=1\); the value
\(\lambda=-1/2\) is not used.

\subsection{A storage newsvendor profit function}
\label{sec:profit_function}

\subsubsection{Economic primitives}
\label{subsec:economic_primitives}

The capacity reservation decision is characterized by the parameters in Table~\ref{tab:notation}.

\begin{table}[H]
\centering
\caption{Main notation.}
\label{tab:notation}
\begin{tabular}{p{0.18\textwidth}p{0.72\textwidth}}
\toprule
Symbol & Meaning \\
\midrule
\(q\) & Reserved or committed storage capacity \\
\(Q\) & Maximum usable capacity \\
\(\xi\) & Random effective opportunity size \\
\(F\) & Distribution function of \(\xi\) \\
\(r\) & Unit net revenue from utilized reserved capacity \\
\(c\) & Unit reservation cost, including degradation or opportunity cost \\
\(s\) & Unit salvage value of unused reserved capacity \\
\(\ell\) & Unit penalty or missed-opportunity loss when capacity is insufficient \\
\(C_o\) & Unit overage cost \\
\(C_u\) & Unit underage cost \\
\bottomrule
\end{tabular}
\end{table}

\begin{assumption}[Integrability and finite risk measures]
\label{ass:integrability}
For every probability distribution \(G\) used to evaluate an objective in this
paper, including the nominal distribution \(F\) and each candidate distribution
in \(\cP\), the effective opportunity satisfies \(\E_G[\xi]<\infty\). Moreover,
for every \(q\in[0,Q]\),
\[
\E_G[|\Pi(q,\xi)|]<\infty,
\qquad
\CVaR_{\alpha,G}(-\Pi(q,\xi))<\infty.
\]
\end{assumption}

\begin{assumption}[Nondegenerate reservation economics]
\label{ass:nondegenerate}
The economic parameters satisfy
\begin{equation*}
r+\ell>c>s.
\end{equation*}
\end{assumption}

Assumption~\ref{ass:integrability} makes all expected-profit, CVaR and regret
objectives below well-defined. Assumption~\ref{ass:nondegenerate} rules out
trivial cases in which reserving no capacity or the maximum capacity is always
optimal independently of uncertainty.

\subsubsection{Profit function}
\label{subsec:profit}

For a capacity decision \(q\) and opportunity realization \(\xi\), define the one-period profit as
\begin{equation}
\Pi(q,\xi)
=
r\min\{q,\xi\}
+s\pos{q-\xi}
-\ell\pos{\xi-q}
-cq.
\label{eq:profit}
\end{equation}
The four terms in~\eqref{eq:profit} represent utilized-capacity revenue, salvage value from unused reserved capacity, missed-opportunity penalty, and reservation cost, respectively.

Equivalently,
\begin{equation*}
\Pi(q,\xi)
=
\begin{cases}
(r+\ell-c)q-\ell \xi, & q\le \xi,\\[3pt]
(r-s)\xi+(s-c)q, & q>\xi.
\end{cases}
\end{equation*}
Thus, for any fixed \(\xi\), \(\Pi(q,\xi)\) is piecewise linear and concave in \(q\). Define the unit overage and underage costs as
\begin{equation*}
C_o=c-s,
\qquad
C_u=r+\ell-c.
\end{equation*}
Under Assumption~\ref{ass:nondegenerate}, \(C_o>0\), \(C_u>0\), and
\begin{equation*}
C_o+C_u=r+\ell-s.
\end{equation*}

\subsection{The finite-candidate distributionally robust tri-objective storage newsvendor problem}
\label{subsec:multi_objective}

The storage operator does not only care about average profit. A policy with high
mean profit may still expose the operator to severe downside losses or perform
poorly when the opportunity distribution is misspecified. We therefore formulate
capacity reservation as a finite-candidate distributionally robust tri-objective
optimization problem rather than as three separate single-objective problems.
The ambiguity set \(\cP\) represents
plausible non-normal opportunity distributions obtained from different fitted
families, sample windows, stress regimes or scenario weights. The mixed-direction
objective is to choose one capacity \(q\) that simultaneously addresses
\begin{equation}
\left\{
\begin{array}{ll}
\displaystyle \max_{0\le q\le Q} &
\displaystyle \inf_{G\in\cP}\E_G[\Pi(q,\xi)], \\[6pt]
\displaystyle \min_{0\le q\le Q} &
\displaystyle \sup_{G\in\cP}\CVaR_{\alpha,G}(-\Pi(q,\xi)), \\[6pt]
\displaystyle \min_{0\le q\le Q} &
\displaystyle \sup_{G\in\cP}
\left\{
V_G^*-\E_G[\Pi(q,\xi)]
\right\}.
\end{array}
\right.
\label{eq:mixed_direction_triobjective}
\end{equation}
Equation~\eqref{eq:mixed_direction_triobjective} is the central decision problem
of this paper. It does not represent three independent optimization problems:
the same capacity \(q\) must be chosen while balancing the three objectives. The
first line protects average economic performance against the least favorable
candidate distribution, the second controls the worst candidate-distribution
tail loss, and the third controls the largest distributional regret. We solve
this mixed-direction problem in three steps. First, we convert it into a common
minimization form. Second, we interpret the resulting vector problem in the
Pareto sense. Third, we compute implementable policies through a scalarized
finite-scenario quadratic program. For the first step, let
\begin{equation*}
J_1(q)
=
-\inf_{G\in\cP}\E_G[\Pi(q,\xi)]
=
\sup_{G\in\cP}\left\{-\E_G[\Pi(q,\xi)]\right\}
\end{equation*}
be the worst-case negative expected profit, so that minimizing \(J_1\) is
equivalent to maximizing the worst-case expected profit. Let
\begin{equation*}
J_2(q)
=
\sup_{G\in\cP}\CVaR_{\alpha,G}(-\Pi(q,\xi))
\end{equation*}
be the worst-case CVaR of the loss \(-\Pi(q,\xi)\). Finally, define the
maximum-regret objective
\begin{equation*}
J_3(q)
=
\sup_{G\in\cP}
\left\{
V_G^*-\E_G[\Pi(q,\xi)]
\right\},
\qquad
V_G^*
=
\max_{0\le z\le Q}\E_G[\Pi(z,\xi)].
\end{equation*}
The mixed-direction problem~\eqref{eq:mixed_direction_triobjective} is therefore
equivalently represented as the tri-objective storage newsvendor problem
\begin{equation}
\min_{0\le q\le Q}
\left(
J_1(q),J_2(q),J_3(q)
\right).
\label{eq:multi_objective_problem}
\end{equation}

Problem~\eqref{eq:multi_objective_problem} should be interpreted in the Pareto
sense. The single-objective distributionally robust expected-profit,
distributionally robust CVaR and maximum-regret models are boundary cases of
this tri-objective formulation.

\begin{definition}[Pareto-efficient reservation decision]
A feasible reservation decision \(q^*\in[0,Q]\) is Pareto efficient for
Problem~\eqref{eq:multi_objective_problem} if there is no other feasible
\(q\in[0,Q]\) such that \(J_i(q)\le J_i(q^*)\) for all \(i=1,2,3\), with at
least one strict inequality.
\end{definition}

In computation, a decision maker can generate robust compromise decisions by
introducing epigraph objective levels \(v=(v_1,v_2,v_3)\) satisfying
\[
v_m\ge J_m(q),\qquad m=1,2,3.
\]
Let \(z_m^{I}\) be an ideal or target value for objective \(m\), and let \(d_m>0\)
be a scaling constant. A standard quadratic achievement scalarization is
\begin{equation}
q^{\lambda}
\in
\argmin
\left\{
\sum_{m=1}^{3}
\lambda_m
\left(\frac{v_m-z_m^{I}}{d_m}\right)^2:
0\le q\le Q,\; v_m\ge J_m(q),\; m=1,2,3
\right\},
\label{eq:quadratic_multiobjective}
\end{equation}
where \(\lambda_m\ge 0\) and \(\sum_{m=1}^{3}\lambda_m=1\). The weights represent
the operator's preference over worst-case mean performance, worst-case tail-risk
protection and regret-based robustness. The scaling constants prevent one
objective from dominating the compromise only because it is measured on a larger
numerical scale. Boundary weights recover the distributionally robust
expected-profit, distributionally robust CVaR, and maximum-regret boundary
policies; interior weights solve the joint
trade-off problem.

\begin{proposition}[Coherent-risk convexity and quadratic scalarization]
\label{prop:multiobjective_convexity}
Under Assumptions~\ref{ass:integrability} and~\ref{ass:nondegenerate},
Problem~\eqref{eq:quadratic_multiobjective} is a convex scalarization of the
finite-candidate distributionally robust tri-objective problem whenever \(z_m^I\) is chosen so that
\(v_m-z_m^I\ge 0\) over the feasible epigraph region. In a finite-scenario
representation, this scalarization can be written as a convex quadratic program.
\end{proposition}

\begin{proof}
For every realization \(\xi\), the profit function \(\Pi(q,\xi)\) is concave in
\(q\). For each \(G\in\cP\), \(-\E_G[\Pi(q,\xi)]\) is convex, and therefore
the supremum defining \(J_1\) is convex. CVaR is a coherent risk measure; in
particular, it is convex, monotone, translation equivariant and positively
homogeneous. Since the loss \(L(q,\xi)=-\Pi(q,\xi)\) is convex in \(q\), the
Rockafellar--Uryasev representation in Section~\ref{sec:cvar} gives a convex
epigraph for each candidate \(G\), and the supremum over \(G\in\cP\) preserves
convexity. For each \(G\in\cP\), the term \(V_G^*-\E_G[\Pi(q,\xi)]\) is a
constant minus a concave function, hence convex in \(q\). The supremum of these
regret functions is convex. The epigraph constraints \(v_m\ge J_m(q)\) are
therefore convex. Because \(v_m-z_m^I\ge0\), the squared normalized deviations in
\eqref{eq:quadratic_multiobjective} are convex and nondecreasing in the relevant
epigraph variables. In a finite-scenario model, the profit and CVaR epigraphs are
linear after introducing standard auxiliary variables, while the objective is a
positive semidefinite quadratic form. Hence the finite-scenario scalarization is
a convex quadratic program.
\end{proof}

Proposition~\ref{prop:multiobjective_convexity} justifies using scalarization to
generate Pareto-efficient candidates without leaving the class of convex
capacity-reservation problems.

The following subsections study the building blocks of the finite-candidate
distributionally robust model. The single-distribution expected-profit problem
gives a critical-fractile benchmark, the CVaR component gives a convex tail-risk
representation, and the regret component gives a robustness criterion against
candidate distributions. These building blocks are then assembled into the
finite-scenario solution method in Section~\ref{sec:finite_scenario_solution}.

\subsection{Distributionally robust expected-profit component and its single-distribution benchmark}
\label{sec:expected_profit}

The first objective of the finite-candidate distributionally robust model is to
maximize worst-case expected profit,
\[
\max_{0\le q\le Q}\inf_{G\in\cP}\E_G[\Pi(q,\xi)].
\]
To obtain a tractable building block for this distributionally robust objective,
first consider a single candidate distribution \(G\). The distribution-specific expected-profit
benchmark chooses \(q\) to maximize
\begin{equation*}
q^{EP}(G)
\in
\argmax_{0\le q\le Q}
\E_G[\Pi(q,\xi)].
\end{equation*}
Using~\eqref{eq:profit}, the objective can be written as
\begin{equation*}
\max_{0\le q\le Q}
\left\{
r\E_G[\min\{q,\xi\}]
+s\E_G[\pos{q-\xi}]
-\ell\E_G[\pos{\xi-q}]
-cq
\right\}.
\end{equation*}
The finite-candidate distributionally robust model in
Section~\ref{sec:finite_scenario_solution}
uses these distribution-specific expected-profit functions inside the epigraph
constraint for worst-case expected profit.

\begin{proposition}[Distribution-specific critical-fractile solution]
\label{prop:critical_fractile}
For \(u\in(0,1)\), let
\[
F^{-1}(u)=\inf\{x\ge 0:F(x)\ge u\}
\]
denote the generalized quantile of a candidate opportunity distribution \(F\).
Suppose Assumptions~\ref{ass:integrability} and~\ref{ass:nondegenerate} hold and \(F\) is continuous. If the unconstrained optimizer is interior, then the expected-profit optimal reservation quantity satisfies
\begin{equation}
F(q^{EP})
=
\frac{r+\ell-c}{r+\ell-s}
=
\frac{C_u}{C_u+C_o}.
\label{eq:critical_fractile}
\end{equation}
With the capacity constraint \(q\in[0,Q]\), an optimal solution is
\begin{equation}
q^{EP}
=
\operatorname{Proj}_{[0,Q]}
\left[
F^{-1}
\left(
\frac{C_u}{C_u+C_o}
\right)
\right].
\label{eq:ep_projected_solution}
\end{equation}
\end{proposition}

\begin{proof}
For points at which the derivative exists,
\[
\frac{d}{dq}\E_F[\Pi(q,\xi)]
=
(r+\ell-c)-(r+\ell-s)F(q).
\]
Setting the derivative to zero yields~\eqref{eq:critical_fractile}. Concavity follows from the piecewise linear concavity of \(\Pi(q,\xi)\) in \(q\), and the constrained solution is obtained by projection onto \([0,Q]\).
\end{proof}

Proposition~\ref{prop:critical_fractile} is the benchmark result against which
the CVaR, regret and numerical policies are compared.

\begin{remark}[Role of non-normality]
Equation~\eqref{eq:ep_projected_solution} shows that the newsvendor model itself does not require normality. Non-normality affects the decision through the quantile function \(F^{-1}\). Asymmetry and heavy tails can shift the relevant critical quantile and generate capacity misallocation under a normal approximation.
\end{remark}

\begin{proposition}[Comparative statics of the expected-profit policy]
\label{prop:comparative_statics}
Suppose Assumption~\ref{ass:nondegenerate} holds and \(F\) is continuous
and strictly increasing around the critical fractile.
Let
\[
\theta=\frac{r+\ell-c}{r+\ell-s}.
\]
Then the interior expected-profit solution \(q^{EP}=F^{-1}(\theta)\) is nondecreasing
in \(r\), \(\ell\), and \(s\), and nonincreasing in \(c\). Moreover, let
\(X_j\) have distribution function \(F_j\), \(j=1,2\). If
\(X_2\succeq_{\mathrm{FSD}}X_1\), equivalently \(F_2(x)\le F_1(x)\) for all \(x\), then
\[
q^{EP}(F_2)\ge q^{EP}(F_1).
\]
\end{proposition}

\begin{proof}
The partial derivatives of \(\theta\) are
\[
\frac{\partial \theta}{\partial r}
=
\frac{c-s}{{(r+\ell-s)}^2}>0,\qquad
\frac{\partial \theta}{\partial \ell}
=
\frac{c-s}{{(r+\ell-s)}^2}>0,
\]
\[
\frac{\partial \theta}{\partial c}
=
-\frac{1}{r+\ell-s}<0,\qquad
\frac{\partial \theta}{\partial s}
=
\frac{r+\ell-c}{{(r+\ell-s)}^2}>0.
\]
Since \(F^{-1}\) is nondecreasing, the first claim follows. For the second
claim, \(X_2\succeq_{\mathrm{FSD}}X_1\) means that
\[
F_2(x)\le F_1(x),\qquad \forall x.
\]
Therefore the generalized quantiles satisfy
\[
F_2^{-1}(\theta)\ge F_1^{-1}(\theta),\qquad \forall \theta\in(0,1).
\]
\end{proof}

Proposition~\ref{prop:comparative_statics} clarifies how storage-market economic
parameters move the expected-profit capacity before risk and regret criteria are
introduced.

\subsection{CVaR tail-risk component}
\label{sec:cvar}

Expected-profit maximization may expose the operator to severe tail losses. Define the loss associated with decision \(q\) as
\begin{equation*}
L(q,\xi)=-\Pi(q,\xi).
\end{equation*}
For a confidence level \(\alpha\in(0,1)\) and a candidate distribution
\(G\in\cP\), the conditional value-at-risk of \(L(q,\xi)\) is
\begin{equation*}
\CVaR_{\alpha,G}(L(q,\xi))
=
\min_{\eta\in\R}
\left\{
\eta+
\frac{1}{1-\alpha}
\E_G[\pos{L(q,\xi)-\eta}]
\right\}.
\end{equation*}

\subsubsection{Distributionally robust CVaR model}
\label{subsec:pure_cvar}

The distributionally robust risk-averse boundary policy solves
\begin{equation*}
q^{RCVaR}
\in
\argmin_{0\le q\le Q}
\sup_{G\in\cP}\CVaR_{\alpha,G}(L(q,\xi)).
\end{equation*}
For a single reference distribution \(F\), this reduces to the standard
Rockafellar--Uryasev reformulation
\begin{equation}
\min_{0\le q\le Q,\,\eta\in\R}
\left\{
\eta+
\frac{1}{1-\alpha}
\E_F[\pos{-\Pi(q,\xi)-\eta}]
\right\}.
\label{eq:pure_cvar_reform}
\end{equation}
CVaR is used here because it is a coherent risk measure. Its convexity and
positive homogeneity make tail-risk control compatible with robust optimization,
while its Rockafellar--Uryasev representation converts the tail expectation into
linear epigraph constraints under a finite scenario representation.

\begin{proposition}[Convexity of the CVaR formulation]
\label{prop:cvar_convexity}
Under Assumptions~\ref{ass:integrability} and~\ref{ass:nondegenerate}, problem~\eqref{eq:pure_cvar_reform} is a convex optimization problem in \((q,\eta)\).
\end{proposition}

\begin{proof}
For every \(\xi\), \(\Pi(q,\xi)\) is concave in \(q\). Hence \(L(q,\xi)=-\Pi(q,\xi)\) is convex in \(q\). The positive-part function is convex and nondecreasing, so \(\pos{L(q,\xi)-\eta}\) is convex in \((q,\eta)\). Taking expectation and adding the linear term \(\eta\) preserves convexity.
\end{proof}

Proposition~\ref{prop:cvar_convexity} shows that adding tail-risk control does
not destroy tractability.

\subsubsection{Sample-average CVaR formulation}
\label{subsec:saa_cvar}

Given samples \(\xi_1,\ldots,\xi_n\), the sample-average approximation of~\eqref{eq:pure_cvar_reform} is
\begin{align}
\min_{q,\eta,u_1,\ldots,u_n}
\quad&
\eta+\frac{1}{(1-\alpha)n}\sum_{i=1}^{n}u_i
\label{eq:saa_cvar}\\
\text{s.t.}\quad&
u_i\ge -\Pi(q,\xi_i)-\eta,\quad i=1,\ldots,n,
\nonumber\\
&u_i\ge 0,\quad i=1,\ldots,n,
\nonumber\\
&0\le q\le Q.
\nonumber
\end{align}
Because \(-\Pi(q,\xi_i)\) is convex and piecewise linear, formulation~\eqref{eq:saa_cvar} is a convex program with one primary capacity variable and auxiliary CVaR variables.

\subsection{Maximum-regret robustness component}
\label{sec:maximum_regret}

Worst-case expected profit and worst-case CVaR protect absolute performance over
the ambiguity set. Regret adds a different robust criterion: it compares a
chosen capacity with the distribution-specific or scenario-specific best
capacity. This is useful in storage markets because price-spike patterns and
activation patterns can change, so a policy that is good for one opportunity
distribution may be far from optimal for another.

\subsubsection{Scenario maximum regret}
\label{subsec:scenario_regret}

Suppose first that only the support
\[
\xi\in[\underline{\xi},\overline{\xi}]
\]
is known. For a realized opportunity \(\xi\), the ex post optimal capacity is \(q^*(\xi)=\min\{\xi,Q\}\). The scenario regret of choosing \(q\) is
\begin{equation*}
R(q,\xi)
=
\Pi(q^*(\xi),\xi)-\Pi(q,\xi).
\end{equation*}
The scenario maximum-regret decision is
\begin{equation*}
q^{SR}
\in
\argmin_{0\le q\le Q}
\sup_{\xi\in[\underline{\xi},\overline{\xi}]}
R(q,\xi).
\end{equation*}

\begin{proposition}[Closed form under interval uncertainty]
\label{prop:scenario_regret_closed_form}
Under Assumption~\ref{ass:nondegenerate}, if \(Q\ge \overline{\xi}\), then
\begin{equation}
R(q,\xi)=C_u\pos{\xi-q}+C_o\pos{q-\xi}.
\label{eq:regret_piecewise_cost}
\end{equation}
The optimal scenario maximum-regret decision is
\begin{equation}
q^{SR}
=
\frac{C_o\underline{\xi}+C_u\overline{\xi}}{C_o+C_u}.
\label{eq:sr_closed_form}
\end{equation}
\end{proposition}

\begin{proof}
When \(Q\ge\overline{\xi}\), the ex post best decision is \(q^*(\xi)=\xi\). Comparing \(\Pi(\xi,\xi)\) with \(\Pi(q,\xi)\) yields~\eqref{eq:regret_piecewise_cost}. The worst-case regret is
\[
\max\{C_o(q-\underline{\xi}),\, C_u(\overline{\xi}-q)\}.
\]
The minimizer equalizes these two terms, which gives~\eqref{eq:sr_closed_form}.
\end{proof}

Proposition~\ref{prop:scenario_regret_closed_form} is stated under
\(Q\ge\overline{\xi}\) to keep the expression transparent. If
\(Q<\overline{\xi}\), define the truncated opportunity
\(\tilde{\xi}=\min\{\xi,Q\}\). The same support-only calculation applies to
\(\tilde{\xi}\). When \(Q\ge \underline{\xi}\), the support endpoints become
\(\underline{\xi}\) and \(\min\{\overline{\xi},Q\}\); if
\(Q<\underline{\xi}\), the truncated opportunity is degenerate at \(Q\).
These endpoints replace \(\underline{\xi}\) and \(\overline{\xi}\)
in~\eqref{eq:sr_closed_form}.

\subsubsection{Distributional maximum regret}
\label{subsec:distributional_regret}

Let \(\cP\) be a set of plausible distributions for the non-normal opportunity variable. For any \(P\in\cP\), define the distribution-specific optimal value
\begin{equation*}
V^*(P)
=
\max_{0\le z\le Q}\E_P[\Pi(z,\xi)].
\end{equation*}
The regret of decision \(q\) under distribution \(P\) is
\begin{equation*}
\mathcal{R}(q;P)
=
V^*(P)-\E_P[\Pi(q,\xi)].
\end{equation*}
The distributional maximum-regret policy solves
\begin{equation*}
q^{MR}
\in
\argmin_{0\le q\le Q}
\sup_{P\in\cP}
\left[
\max_{0\le z\le Q}\E_P[\Pi(z,\xi)]
-\E_P[\Pi(q,\xi)]
\right].
\end{equation*}
This formulation follows the minimax-regret logic of partial-information newsvendor models~\citep{PerakisRoels2008}, but adapts it to storage-capacity reservation.

\subsubsection{Candidate non-normal distributions}
\label{subsec:candidate_distributions}

The ambiguity set \(\cP\) can be specified in several ways:
\begin{enumerate}[label = (\roman*)]
\item \textit{Parametric NMVM candidates}: degenerate NMVM, finite-regime NMVM
mixtures, VG-NMVM, GH-NMVM, or empirical NMVM scenario mixtures.
\item \textit{Moment and tail information}: distributions satisfying constraints on mean, variance, support, tail quantiles, skewness direction, or tail probability.
\item \textit{Data-driven neighborhoods}: distributions close to an empirical distribution \(\widehat P_n\) under Wasserstein distance, CDF bands, stochastic dominance bands, or moment confidence sets.
\end{enumerate}

\subsubsection{Finite-candidate formulation}
\label{subsec:finite_candidate_mr}

If \(\cP=\{P_1,\ldots,P_K\}\), compute
\begin{equation*}
V_k^*
=
\max_{0\le z\le Q}\E_{P_k}[\Pi(z,\xi)],
\qquad k=1,\ldots,K.
\end{equation*}
The maximum-regret problem becomes
\begin{align}
\min_{q,t}\quad & t
\label{eq:finite_mr}\\
\text{s.t.}\quad&
t\ge V_k^*-\E_{P_k}[\Pi(q,\xi)],
\quad k=1,\ldots,K,
\nonumber\\
&0\le q\le Q.
\nonumber
\end{align}

\begin{proposition}[Convexity of finite-candidate maximum regret]
\label{prop:finite_mr_convexity}
Under Assumption~\ref{ass:integrability}, problem~\eqref{eq:finite_mr} is a two-variable epigraph convex program with one primary capacity variable.
\end{proposition}

\begin{proof}
For each \(k\), \(\E_{P_k}[\Pi(q,\xi)]\) is concave in \(q\), so \(V_k^*-\E_{P_k}[\Pi(q,\xi)]\) is convex. The pointwise maximum of finitely many convex functions is convex, and the epigraph formulation~\eqref{eq:finite_mr} preserves convexity.
\end{proof}

Proposition~\ref{prop:finite_mr_convexity} provides the finite-candidate robust
counterpart to the closed-form interval benchmark.

If each \(P_k\) is represented by weighted scenarios
\(\left\{(\xi_i,p_{ki}):i=1,\ldots,n\right\}\), then
\begin{equation*}
\E_{P_k}[\Pi(q,\xi)]
=
\sum_{i=1}^{n}p_{ki}\Pi(q,\xi_i),
\end{equation*}
and~\eqref{eq:finite_mr} can be solved directly from sampled or bootstrapped distributions.

\subsection{Sample-based construction of finite-candidate ambiguity sets}
\label{sec:sample_based}

In an empirical implementation, observations or simulated scenarios
\(\xi_1,\ldots,\xi_n\) are not used only to form one empirical distribution.
They are used to construct a finite ambiguity set
\[
\cP_n=\{P_1,\ldots,P_K\}.
\]
The candidates may correspond to different fitted non-normal families, different
rolling sample windows, bootstrap reweightings, or stress distributions for
price-spike and activation regimes. Each candidate \(P_k\) is represented on a
common scenario support by probability weights \(p_{ki}\). This construction
keeps the solution distributionally robust in the finite-candidate sense: the
capacity is chosen against all candidate distributions rather than against a
single empirical CDF.

\subsection{Solution method for the scalarized finite-candidate distributionally robust model}
\label{sec:finite_scenario_solution}

The three distributionally robust objective components can be solved jointly once
the opportunity distribution is represented by a finite candidate set on a common
scenario support. Let
\(\left\{\xi_i:i=1,\ldots,n\right\}\) be a common scenario set. Each candidate
distribution \(P_k\in\cP\), \(k=1,\ldots,K\), is represented by probability
weights \(\left\{p_{ki}:i=1,\ldots,n\right\}\):
\[
p_{ki}\ge 0,\quad \sum_{i=1}^{n}p_{ki}=1,\quad k=1,\ldots,K.
\]
For every
candidate \(P_k\), first compute its distribution-specific best value
\begin{equation*}
V_k^*
=
\max_{0\le z\le Q}
\sum_{i=1}^{n}p_{ki}\Pi(z,\xi_i).
\end{equation*}
For each scenario \(i\), write the loss \(L_i(q)=-\Pi(q,\xi_i)\) as the maximum
of two affine functions:
\[
L_i(q)
=
\max\left\{
\ell\xi_i-C_u q,\;
C_o q-(r-s)\xi_i
\right\}.
\]
Introduce a loss epigraph variable \(h_i\) satisfying
\[
h_i\ge \ell\xi_i-C_u q,\qquad
h_i\ge C_o q-(r-s)\xi_i,\qquad i=1,\ldots,n.
\]
Let \(z_1^I,z_2^I,z_3^I\) be ideal or target objective values and let
\(d_1,d_2,d_3>0\) be scaling constants. The scalarized finite-candidate
distributionally robust reservation decision can then be written as the
following finite-scenario quadratic program:
\begin{align}
\min_{q,h_i,\rho,\chi,t,\eta_k,u_{ki},\widehat\rho,\widehat\chi,\widehat t}\quad&
\lambda_1\widehat\rho^2+\lambda_2\widehat\chi^2+\lambda_3\widehat t^2
\label{eq:finite_qp_problem}\\
\text{s.t.}\quad&
h_i\ge \ell\xi_i-C_u q,\quad i=1,\ldots,n,
\nonumber\\
&h_i\ge C_o q-(r-s)\xi_i,\quad i=1,\ldots,n,
\nonumber\\
&\rho\ge \sum_{i=1}^{n}p_{ki}h_i,
\quad k=1,\ldots,K,
\nonumber\\
&\chi\ge
\eta_k+\frac{1}{1-\alpha}\sum_{i=1}^{n}p_{ki}u_{ki},
\quad k=1,\ldots,K,
\nonumber\\
&u_{ki}\ge h_i-\eta_k,
\quad i=1,\ldots,n,\quad k=1,\ldots,K,
\nonumber\\
&u_{ki}\ge 0,
\quad i=1,\ldots,n,\quad k=1,\ldots,K,
\nonumber\\
&t\ge V_k^*+\sum_{i=1}^{n}p_{ki}h_i,
\quad k=1,\ldots,K,
\nonumber\\
&\widehat\rho\ge \frac{\rho-z_1^I}{d_1},\quad
\widehat\chi\ge \frac{\chi-z_2^I}{d_2},\quad
\widehat t\ge \frac{t-z_3^I}{d_3},
\nonumber\\
&\widehat\rho,\widehat\chi,\widehat t\ge 0,\qquad
0\le q\le Q.
\nonumber
\end{align}
Problem~\eqref{eq:finite_qp_problem} is the computational form of the paper's
main finite-candidate distributionally robust tri-objective model. The variable \(\rho\) is the epigraph of the
worst-case negative expected profit, \(\chi\) is the epigraph of the worst-case
CVaR loss, and \(t\) is the epigraph of the maximum regret relative to the
distribution-specific best decision. The variables
\(\widehat\rho,\widehat\chi,\widehat t\) are normalized deviations from the
chosen ideal or target levels. All constraints in
\eqref{eq:finite_qp_problem} are linear, and the objective is a positive
semidefinite quadratic function. Hence the joint finite-candidate
distributionally robust expected-profit--CVaR--regret problem admits a standard
convex quadratic programming representation under a finite scenario ambiguity
set.

Operationally, the finite-candidate distributionally robust solution procedure is
as follows. First, construct a
common scenario support for the effective opportunity \(\xi\). Second, construct
the ambiguity set \(\cP=\{P_1,\ldots,P_K\}\) by assigning candidate probability
weights \(p_{ki}\) to the same scenario support. Third, compute the
distribution-specific best values \(V_k^*\) and choose ideal or target levels
\((z_1^I,z_2^I,z_3^I)\) with scales \((d_1,d_2,d_3)\). Fourth, choose the
preference weights \((\lambda_1,\lambda_2,\lambda_3)\) and solve
\eqref{eq:finite_qp_problem}. Finally, repeat the procedure over a grid of weight
vectors to approximate the finite-candidate distributionally robust Pareto
frontier of capacity-reservation decisions. Robustness is retained at the solution stage because every epigraph
constraint is imposed for all \(k=1,\ldots,K\).

\section{Closed-form solutions and analytical implications}
\label{sec:closed_form_solutions}

This section derives analytical benchmark solutions used to interpret the
finite-candidate distributionally robust model. The objective is not to solve the
full distributionally robust tri-objective problem in
closed form. That problem is solved by the finite-scenario epigraph formulation
in Section~\ref{sec:finite_scenario_solution}. Instead, the closed-form
benchmarks identify why non-normality and regret matter: the single-distribution
expected-profit benchmark depends on a distributional quantile, whereas the
support-only regret benchmark depends on the support interval and the
overage--underage cost trade-off.

\subsection{Expected-profit closed-form solution}
\label{subsec:closed_form_ep}

\begin{proposition}[Expected-profit solution and normal approximation]
\label{prop:closed_form_ep}
Suppose Assumptions~\ref{ass:integrability} and~\ref{ass:nondegenerate} hold and \(F\) is continuous. Let
\begin{equation*}
\theta
=
\frac{C_u}{C_u+C_o}
=
\frac{r+\ell-c}{r+\ell-s}.
\end{equation*}
Then the expected-profit maximizing capacity is
\begin{equation}
q^{EP}
=
\operatorname{Proj}_{[0,Q]}\left[F^{-1}(\theta)\right].
\label{eq:closed_form_ep}
\end{equation}
If the opportunity distribution is approximated by a normal distribution with
the same mean \(\mu\) and standard deviation \(\sigma\), the corresponding
moment-matched normal policy is
\begin{equation}
q^{N}
=
\operatorname{Proj}_{[0,Q]}
\left[
\mu+\sigma\Phi^{-1}(\theta)
\right],
\label{eq:normal_benchmark}
\end{equation}
where \(\Phi\) is the standard normal distribution function. The induced
normal-approximation capacity distortion is
\begin{equation*}
\Delta_q^N
=
q^{N}-q^{EP}.
\end{equation*}
\end{proposition}

\begin{proof}
The derivative of the expected profit at points where it exists is
\[
\frac{d}{dq}\E_F[\Pi(q,\xi)]
=
C_u-(C_u+C_o)F(q).
\]
Since \(\Pi(q,\xi)\) is concave in \(q\) for every fixed \(\xi\),
\(\E_F[\Pi(q,\xi)]\) is concave in \(q\). Therefore the unconstrained optimizer
satisfies \(F(q)=C_u/(C_u+C_o)\). Projection onto \([0,Q]\)
gives~\eqref{eq:closed_form_ep}. Formula~\eqref{eq:normal_benchmark} follows by
replacing \(F^{-1}\) with the quantile function of \(N(\mu,\sigma^2)\).
\end{proof}

\begin{remark}
Proposition~\ref{prop:closed_form_ep} shows that the newsvendor solution is
already distribution-dependent. The normal model is only one possible choice of
\(F\). Therefore the main issue in the present paper is not whether a normal
distribution can be inserted into the formula, but how much the decision changes
when the critical quantile of a non-normal opportunity distribution is replaced
by a moment-matched normal quantile.
\end{remark}

\subsection{Closed-form expected-profit loss from capacity distortion}
\label{subsec:closed_form_profit_loss}

\begin{proposition}[Expected-profit loss identity]
\label{prop:closed_form_profit_loss}
Under the conditions of Proposition~\ref{prop:closed_form_ep}, the expected
profit loss from using a feasible capacity \(q\) instead of \(q^{EP}\) is
\begin{equation}
\Delta_{\Pi}(q)
=
\E_F[\Pi(q^{EP},\xi)]-\E_F[\Pi(q,\xi)]
=
\int_{\min\{q,q^{EP}\}}^{\max\{q,q^{EP}\}}
\left|C_u-(C_u+C_o)F(z)\right|\,dz.
\label{eq:closed_form_profit_loss}
\end{equation}
\end{proposition}

\begin{proof}
Let \(H(q)=\E_F[\Pi(q,\xi)]\). From the proof of
Proposition~\ref{prop:closed_form_ep}, \(H\) is absolutely continuous on
\([0,Q]\) and its derivative exists almost everywhere. At such points,
\[
H'(q)=C_u-(C_u+C_o)F(q)
\]
with the same expression serving as the continuous subgradient representative
when \(F\) is continuous. Since \(H\) is concave and \(q^{EP}\) maximizes
\(H\), the derivative is nonnegative to the left of \(q^{EP}\) and nonpositive
to the right of \(q^{EP}\). Applying the fundamental theorem of calculus between
\(q\) and \(q^{EP}\) gives~\eqref{eq:closed_form_profit_loss}.
\end{proof}

\begin{remark}
The identity in Proposition~\ref{prop:closed_form_profit_loss} is useful for the
numerical analysis because it converts a distributional approximation error into
an economic loss. It also explains why errors around the critical fractile are
more important than errors around the mean when \(\theta\) is high.
\end{remark}

\subsection{Closed-form scenario maximum-regret solution}
\label{subsec:closed_form_regret}

\begin{proposition}[Support-only minimax regret]
\label{prop:closed_form_regret}
Suppose Assumption~\ref{ass:nondegenerate} holds, only the support interval
\(\xi\in[\underline{\xi},\overline{\xi}]\) is known and
\(Q\ge \overline{\xi}\). Then the scenario regret is
\begin{equation}
R(q,\xi)
=
C_u\pos{\xi-q}+C_o\pos{q-\xi}.
\label{eq:closed_form_scenario_regret}
\end{equation}
The closed-form minimax-regret capacity is
\begin{equation}
q^{SR}
=
\frac{C_o\underline{\xi}+C_u\overline{\xi}}{C_o+C_u}.
\label{eq:closed_form_sr}
\end{equation}
The corresponding worst-case regret is
\begin{equation}
\mathcal R^{SR}
=
\frac{C_o C_u}{C_o+C_u}
\left(\overline{\xi}-\underline{\xi}\right).
\label{eq:closed_form_sr_value}
\end{equation}
\end{proposition}

\begin{proof}
When \(Q\ge \overline{\xi}\), the ex post optimal capacity is \(q^*(\xi)=\xi\).
Comparing \(\Pi(\xi,\xi)\) and \(\Pi(q,\xi)\)
gives~\eqref{eq:closed_form_scenario_regret}. The worst-case regret over the interval
is
\[
\max\{C_o(q-\underline{\xi}),\, C_u(\overline{\xi}-q)\}.
\]
The minimizer equalizes the two terms, which gives~\eqref{eq:closed_form_sr}.
Substituting~\eqref{eq:closed_form_sr} into either term
gives~\eqref{eq:closed_form_sr_value}.
\end{proof}

\begin{remark}
Proposition~\ref{prop:closed_form_regret} is deliberately conservative. Unlike
the critical-fractile policy, it does not use the full distribution \(F\). This
makes it a useful benchmark when the support of the opportunity variable is more
credible than a fitted non-normal distribution.
\end{remark}

\subsection{Analytical implications of non-normality}
\label{subsec:analytical_implications}

The above results imply four testable effects.
\begin{enumerate}[label = (\alph*)]
\item \textit{Quantile effect}. Non-normality affects the risk-neutral decision
through \(F^{-1}(\theta)\), not through the mean alone.
\item \textit{Tail-risk effect}. Heavy-tailed opportunity distributions affect
CVaR and regret more strongly than expected profit because these objectives
evaluate downside losses and worst-case deviations.
\item \textit{Normal-approximation effect}. The difference
between~\eqref{eq:normal_benchmark} and~\eqref{eq:closed_form_ep} measures how much
capacity is misallocated when a non-normal opportunity is summarized only by its
mean and variance.
\item \textit{Local-CDF effect}. The finite-sample difficulty of estimating
\(q^{EP}\) depends on the behavior of \(F\) around \(q^{EP}\). Suppose there exist
\(\delta>0\), \(\gamma>0\), and \(\beta\ge 0\) such that, for all
\(|q-q^{EP}|\le\delta\),
\[
|F(q)-F(q^{EP})|
\ge
\gamma |q-q^{EP}|^{\beta+1}.
\]
Since \(F(q^{EP})=\theta\), this local growth condition implies
\begin{equation*}
|q-q^{EP}|
\le
\gamma^{-1/(\beta+1)}
|F(q)-\theta|^{1/(\beta+1)}
\end{equation*}
for \(q\) in the same neighborhood. Thus, a flatter CDF around the critical
fractile, represented by a smaller \(\gamma\) or a larger \(\beta\), amplifies
the capacity error induced by empirical CDF error.
\end{enumerate}

\begin{remark}
The distributionally robust CVaR and scalarized finite-candidate multi-objective
models do not admit comparable closed-form solutions under arbitrary non-normal
ambiguity sets. This is not a defect of the formulation. It is precisely why the
closed-form policies above are used as benchmarks and the full finite-candidate
distributionally robust model is solved through the finite-scenario quadratic
programming formulation in
Section~\ref{sec:finite_scenario_solution}.
\end{remark}

\section{Numerical analysis}
\label{sec:numerical_analysis}

This section evaluates the proposed model with public RTE balancing-market data.
It follows one empirical decision path. First, we construct an effective aFRR
storage-opportunity series and document its non-normality. Second, we fit a
finite set of log-NMVM candidate distributions and examine the critical capacity
implied by each candidate. Third, we solve the finite-candidate robust policy
comparison for expected profit, CVaR, regret and their multi-objective
compromise. Finally, we vary the critical fractile to show how the economic
trade-off changes the normal-approximation error. Consider a
storage operator that reserves regulation capacity before an operating window.
The reserved capacity \(q\), measured in MW, can be used for ancillary-service
activation, imbalance reduction or related flexibility services. The random
variable \(\xi\), also measured in MW-equivalent units, denotes the effective
amount of storage capacity that can be profitably used during the window. If
\(\xi>q\), the operator misses part of the market opportunity; if \(q>\xi\),
part of the reserved storage capacity remains idle.

\subsection{RTE aFRR data and construction of the opportunity variable}
\label{subsec:rte_descriptive}

The empirical description uses the RTE Balancing Energy API, specifically the
\texttt{standard\_afrr\_data} endpoint, from 1 January 2024 to 1 January 2026.
The raw observations are reported at 15-minute resolution. For each interval
\(t\), let \(A_t^+\) and \(A_t^-\) denote the upward and downward aFRR activated
volumes in France, and let \(\pi_t^+\) and \(\pi_t^-\) denote the corresponding
weighted-average activation prices. Define the activated aFRR opportunity
\[
A_t=A_t^+ + A_t^-,
\]
and the price-intensity index
\[
P_t=
\frac{
A_t^+|\pi_t^+|+A_t^-|\pi_t^-|
}{
A_t^+ + A_t^-
}.
\]
The price-weighted effective opportunity is then
\[
\xi_t=A_t\frac{P_t}{\bar P},
\]
where \(\bar P\) is the sample average of positive \(P_t\). We report
\(\tilde{\xi}_t=\xi_t/1000\) in GW-equivalent units. For daily and weekly
reservation frequencies, the capacity opportunity is aggregated by the
within-period peak,
\[
\tilde{\xi}^{day}_d=\max_{t\in d}\tilde{\xi}_t,
\qquad
\tilde{\xi}^{week}_w=\max_{t\in w}\tilde{\xi}_t,
\]
because \(q\) is a capacity reservation and therefore must cover the relevant
peak opportunity inside the reservation window.

\begin{table}[H]
\centering
\caption{Descriptive statistics of RTE aFRR effective opportunity under three reservation frequencies.}
\label{tab:rte_descriptive_statistics}
\small
\resizebox{\textwidth}{!}{%
\begin{tabular}{lrrrrrrrrrrrrr}
\toprule
Frequency & $N$ & Mean & SD & $Q_{1}$ & $Q_{5}$ & $Q_{10}$ & Median & $Q_{90}$ & $Q_{95}$ & $Q_{99}$ & Max & Skew. & Kurt. \\
\midrule
15-min & 70,176 & 0.411 & 1.081 & 0.001 & 0.006 & 0.012 & 0.123 & 0.911 & 1.710 & 4.735 & 49.535 & 12.516 & 331.583 \\
Daily & 731 & 4.470 & 4.333 & 0.771 & 1.071 & 1.264 & 3.326 & 8.576 & 11.425 & 20.403 & 49.535 & 4.343 & 34.400 \\
Weekly & 105 & 8.710 & 7.494 & 1.954 & 2.398 & 2.783 & 6.897 & 13.783 & 21.260 & 40.410 & 49.535 & 3.121 & 14.955 \\
\bottomrule
\end{tabular}
}
\begin{flushleft}\footnotesize Notes: The variable is reported in GW-equivalent units, computed as $\xi/1000$, where $\xi$ is the price-weighted MW-equivalent aFRR opportunity. For daily and weekly frequencies, $\xi$ is the maximum 15-minute effective opportunity within the corresponding period. The lower-tail quantiles $Q_1,Q_5,Q_{10}$ and upper-tail quantiles $Q_{90},Q_{95},Q_{99}$ are reported to show asymmetry.\end{flushleft}
\end{table}

Table~\ref{tab:rte_descriptive_statistics} shows strong right skewness at all
three frequencies. The 15-minute series is the most extreme, with skewness above
12 and kurtosis above 300. The daily and weekly aggregations reduce but do not
remove the asymmetry. In all cases, the upper-tail quantiles are much farther
from the median than the lower-tail quantiles. For example, the 95th percentile
is about 13.9 times the median in the 15-minute series, and remains more than
three times the median after daily or weekly aggregation. This matters for a
newsvendor-type capacity decision because the expected-profit benchmark is a
critical-quantile rule, while the robust CVaR and regret components are also
driven by tail outcomes. The empirical evidence is therefore used to document
the shape of storage opportunities, not to calibrate a site-specific revenue
model.

\begin{figure}[t!]
\centering
\includegraphics[width=0.95\textwidth]{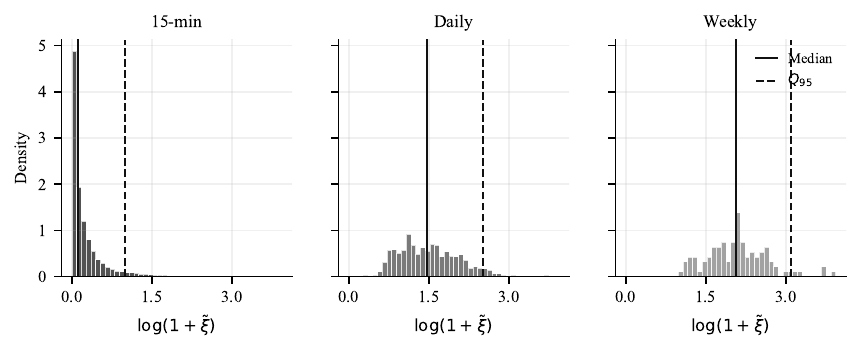}
\caption{Empirical distributions of the effective aFRR opportunity under
15-minute, daily and weekly reservation frequencies. The horizontal axis uses
\(\log(1+\tilde{\xi})\), where \(\tilde{\xi}=\xi/1000\). Solid and dashed
vertical lines mark the median and the 95th percentile, respectively.}
\label{fig:rte_distribution}
\end{figure}

\begin{figure}[t!]
\centering
\includegraphics[width=0.95\textwidth]{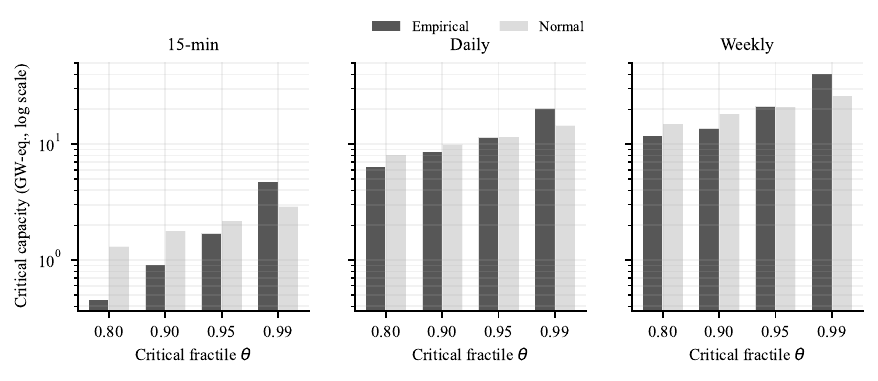}
\caption{Critical-fractile distortion under a moment-matched normal
approximation. The empirical critical capacities are compared with the
corresponding quantiles of a normal distribution matched to the sample mean and
standard deviation.}
\label{fig:rte_quantile_distortion}
\end{figure}

\subsection{Critical-capacity distortion and empirical NMVM candidates}
\label{subsec:rte_nmvm_fit}

The descriptive evidence above shows non-normality, but the newsvendor decision
depends on a specific functional of the distribution: the critical capacity
\(F^{-1}(\theta)\). We therefore fit a finite set of candidate distributions to
the positive 15-minute RTE opportunity series. The candidate set contains the
empirical distribution, a lognormal benchmark, two finite-regime log-NMVM
candidates with three fixed positive stress states, and continuous log-VG and
log-GH candidates. The finite-regime probabilities and latent parameters are
estimated by EM; the log-VG and log-GH candidates are estimated on a
deterministic quantile grid of \(\log \xi\). The direct normal model on \(\xi\)
is retained only as a benchmark policy, not as a robust candidate distribution.

The storage economics used in the empirical decision comparison are
\[
r=8,\qquad c=4,\qquad s=0,\qquad \ell=12,\qquad Q=25,
\]
which imply \(C_o=4\), \(C_u=16\), and
\(\theta=C_u/(C_o+C_u)=0.8\). Because \(\tilde{\xi}=\xi/1000\), the capacities
reported below are in GW-equivalent units and can be multiplied by \(1000\) to
obtain MW-equivalent capacities.

\begin{example}[Critical-capacity distortion]
\label{ex:quantile_distortion}
Table~\ref{tab:rte_15min_em_nmvm_critical_fractile} compares the empirical
\(0.8\)-critical capacity with the capacities implied by the fitted candidates.

\begin{table}[H]
\centering
\caption{Critical-fractile accuracy of EM-estimated log-NMVM candidates under RTE 15-minute aFRR opportunities, reported in GW-equivalent units.}
\label{tab:rte_15min_em_nmvm_critical_fractile}
\small
\begin{tabular}{lrrr}
\toprule
Candidate & $q_{0.8}$ & Error & Abs. error \\
\midrule
Empirical & 0.458 & 0.000 & 0.000 \\
Weibull-shaped log-NMVM & 0.468 & 0.010 & 0.010 \\
Gamma-shaped log-NMVM & 0.471 & 0.014 & 0.014 \\
Log-VG NMVM & 0.444 & -0.014 & 0.014 \\
Log-GH NMVM & 0.439 & -0.019 & 0.019 \\
Lognormal & 0.478 & 0.020 & 0.020 \\
Normal on xi & 1.322 & 0.864 & 0.864 \\
\bottomrule
\end{tabular}
\end{table}

The empirical \(0.8\)-critical capacity is \(0.458\) GW-equivalent. The direct
normal model on \(\xi\) gives \(1.322\) GW-equivalent, over-reserving by
\(0.864\) GW-equivalent. By contrast, the fitted log-NMVM candidates remain close
to the empirical critical capacity: the finite-regime, log-VG and log-GH errors
are all about \(0.02\) GW-equivalent or less. Thus, the empirical data support
the paper's central claim in decision terms. Normality is not merely a poor
descriptive fit; it can move the storage capacity chosen by the critical
fractile rule, while the NMVM candidate family preserves the relevant
non-normal shape more accurately.
\end{example}

\subsection{Finite-candidate robust policy comparison}
\label{subsec:numerical_design}

We next evaluate the finite-candidate robust policies on the empirical
candidate set. All expectations, CVaR values and regrets are computed on
deterministic quantile grids built from the empirical distribution and the
fitted log-NMVM candidates. This avoids Monte Carlo noise in the tables and
makes policy differences attributable to the decision rule rather than to random
simulation error. CVaR is evaluated at \(\alpha=0.90\).

\begin{example}[Robust policy trade-offs]
\label{ex:policy_tradeoff}
Let \(\cP=\{P_1,\ldots,P_K\}\) denote the empirical finite-candidate ambiguity
set. For a capacity \(q\), the reported policy comparison uses
\[
\inf_{P_k\in\cP}\E_{P_k}[\Pi(q,\xi)],\qquad
\sup_{P_k\in\cP}\CVaR_{0.90,P_k}(-\Pi(q,\xi)),
\]
and
\[
\max_{P_k\in\cP}\left\{V_k^*-\E_{P_k}[\Pi(q,\xi)]\right\}
\]
as the worst-case expected profit, worst-case tail loss and maximum regret.
Unmet opportunity is \(\sup_{P_k\in\cP}\E_{P_k}[\pos{\xi-q}]\), and idle
capacity is \(\sup_{P_k\in\cP}\E_{P_k}[\pos{q-\xi}]\). The multi-objective
policy uses the same normalized quadratic compromise form as
\eqref{eq:finite_qp_problem}, with equal weights on expected-profit loss, CVaR
loss and regret.

\begin{table}[H]
\centering
\caption{Empirical finite-candidate distributionally robust policy comparison under EM-fitted RTE 15-minute log-NMVM candidates, reported in GW-equivalent units.}
\label{tab:rte_15min_em_nmvm_fcdr_policy_comparison}
\small
\begin{tabular}{lrrrrrr}
\toprule
Policy & $q$ & $\inf\E[\Pi]$ & $\sup$ CVaR & Regret & Unmet & Idle \\
\midrule
Normal rule & 1.322 & -9.085 & 59.763 & 1.521 & 0.504 & 1.040 \\
Empirical EP & 0.458 & -7.631 & 68.453 & 0.001 & 0.604 & 0.276 \\
Weibull-shaped EP & 0.468 & -7.633 & 68.291 & 0.003 & 0.602 & 0.284 \\
FC-DR EP & 0.438 & -7.630 & 68.775 & 0.006 & 0.608 & 0.260 \\
FC-DR CVaR & 2.862 & -13.727 & 56.893 & 6.501 & 0.428 & 2.516 \\
FC max regret & 0.455 & -7.631 & 68.506 & 0.002 & 0.605 & 0.274 \\
FC-DR MO & 1.498 & -9.546 & 58.987 & 2.006 & 0.492 & 1.204 \\
\bottomrule
\end{tabular}
\begin{flushleft}
\footnotesize
\textit{Notes.} RTE denotes R\'eseau de Transport d'\'Electricit\'e. EM denotes expectation--maximization, NMVM denotes normal mean--variance mixture, FC-DR denotes finite-candidate distributionally robust, EP denotes expected-profit, MO denotes multi-objective, and CVaR denotes conditional value-at-risk. The capacity \(q\), Unmet and Idle are reported in GW-equivalent units. Unmet is \(\sup_{P_k\in\mathcal{P}}\mathbb{E}_{P_k}[(\xi-q)^+]\), and Idle is \(\sup_{P_k\in\mathcal{P}}\mathbb{E}_{P_k}[(q-\xi)^+]\).
\end{flushleft}
\end{table}

The RTE-based policy comparison shows the tri-objective trade-off directly. The
normal rule reserves far more capacity than the empirical and EM-fitted
expected-profit policies; it also has lower worst-case expected profit and
larger regret over the empirical candidate set. The FC-DR expected-profit and
maximum-regret policies stay close to the empirical critical capacity. The
FC-DR CVaR policy chooses a much larger capacity; it reduces worst-case tail
loss and unmet opportunity, but it gives up worst-case expected profit and
creates more idle reserved capacity. The finite-candidate multi-objective policy
lies between these boundary solutions, reflecting the compromise among mean
performance, tail protection and regret robustness.
\end{example}

\subsection{Sensitivity to the critical fractile}
\label{subsec:numerical_sensitivity}

Finally, Table~\ref{tab:numerical_theta_sensitivity} varies the critical
fractile on the empirical RTE 15-minute opportunity distribution. In storage
terms, a larger \(\theta\) corresponds to a more expensive missed market
opportunity or a lower cost of idle reserved capacity. The table compares the
empirical critical capacity with the critical capacity implied by a direct
normal fit on \(\xi\).

\begin{table}[H]
\centering
\caption{RTE storage reservation distortion under a direct normal approximation
at different critical fractiles, reported in GW-equivalent units.}
\label{tab:numerical_theta_sensitivity}
\begin{tabular}{rrrr}
\toprule
\(\theta\) & \(q^{EP}=F^{-1}(\theta)\) & \(q^N\) & \(q^N-q^{EP}\) \\
\midrule
0.6 & 0.188 & 0.685 & 0.497 \\
0.7 & 0.286 & 0.978 & 0.692 \\
0.8 & 0.458 & 1.322 & 0.864 \\
0.9 & 0.913 & 1.797 & 0.884 \\
\bottomrule
\end{tabular}
\end{table}

Table~\ref{tab:numerical_theta_sensitivity} adds a second empirical message.
The normal approximation over-reserves capacity across the reported critical
fractiles, and the absolute error grows as the decision moves toward the upper
tail. Thus, the storage economic parameters \((r,c,s,\ell)\), which determine
\(\theta\), interact with the empirical shape of \(\xi\). Non-normality is not a
generic descriptive issue; it changes the capacity decision most strongly when
the economic trade-off places more weight on avoiding missed high-opportunity
intervals.

Overall, the numerical analysis supports the central modeling claim: under
non-normal uncertainty, storage capacity reservation should be evaluated at the
level of decision quality rather than distributional fit alone. The RTE evidence
shows that empirical storage-opportunity data can generate large
normal-approximation errors at the critical capacity and that EM-fitted
log-NMVM candidates can closely reproduce the empirical critical capacity. The
empirical finite-candidate policy comparison also shows that policies close in
worst-case expected profit can differ substantially in worst-case tail loss,
maximum regret, unmet storage opportunity and idle capacity. These conclusions
are conditional on the specified RTE opportunity construction and economic
parameters. A full market-policy evaluation would take the next step: estimate
candidate distributions on a rolling training window and repeat the
finite-candidate distributionally robust profit, CVaR and regret evaluation on
out-of-sample market windows.

\section{Conclusion}
\label{sec:conclusion}

This paper has developed a finite-candidate distributionally robust
tri-objective newsvendor framework for energy-storage capacity reservation under
non-normal uncertainty and distributional ambiguity. The model has
converted uncertain market opportunities into an effective opportunity variable
and used a unified profit function to compare excess-capacity and
insufficient-capacity costs. On this basis, worst-case expected profit,
worst-case CVaR tail loss and maximum regret have been formulated as one joint
finite-candidate distributionally robust decision problem rather than as
unrelated decision criteria. We have shown
that the resulting model contains tractable critical-fractile, coherent CVaR and
finite-candidate regret components and admits a finite-scenario quadratic
programming form that solves the three distributionally robust objectives
together. The closed-form benchmarks have
clarified how non-normality changes the critical capacity quantile and how normal
approximation creates capacity distortion. The NMVM specification has provided a
positive and flexible non-normal opportunity family that retains market-stress
mixing, skewness and tail behavior that a moment-matched normal benchmark
removes. The numerical comparison has therefore shown not merely that the data
are non-normal, but that the omitted shape information changes the capacity
quantile and the finite-candidate distributionally robust
profit--CVaR--regret evaluation. We have also compared
profit-seeking, tail-risk-averse and regret-robust policies under the same
evening-window storage reservation structure.

\end{document}